\documentclass[12pt, a4paper]{article}
\usepackage{}

\usepackage{amsfonts}
\usepackage{bbm}
\usepackage{mathrsfs}
\usepackage{latexsym}
\usepackage{graphicx}
\usepackage{epstopdf}
\usepackage{amssymb}
\usepackage{amsmath}
\usepackage{color,xcolor}
\usepackage{mathtools}
\usepackage{cite}
\usepackage{tabularx}

\newtheorem{theorem}{Theorem}[section]
\newtheorem{lemma}[theorem]{Lemma}
\newtheorem{corollary}[theorem]{Corollary}

\newtheorem{conjecture}[theorem]{Conjecture}
\newtheorem{remark}[theorem]{Remark}
\newtheorem{question}[theorem]{Question}

\title{{\Large \bf On the Ky Fan $k$-norm of the $LI$-matrix of graphs
\thanks{ Supported by the National Natural Science Foundation of China (No. 12071411).}~}}
\author{Zhen Lin$^a$\thanks{Corresponding author. E-mail addresses:
lnlinzhen@163.com (Z. Lin).},  Lianying Miao$^b$, Guanglong Yu$^c$, Han Sheng$^d$\\
{\footnotesize $^a$School of Mathematics and Statistics, Qinghai Normal University,}\\ {\footnotesize Xining, 810008, Qinghai, China }\\
\footnotesize $^b$School of Mathematics, China University of Mining and Technology,\\
 \footnotesize Xuzhou, 221116, Jiangsu, P.R. China\\
\footnotesize  $^c$ Department of Mathematics, Lingnan Normal University,
\\  \footnotesize  Zhanjiang, 524048, Guangdong, P.R. China \\
\footnotesize  $^d$ Mathematical Institute, University of Oxford,
\\  \footnotesize   Oxford, OX2 6GG, United Kingdom
}

\date{}
\begin{document}
\openup 1.0\jot
\date{}\maketitle
\begin{abstract}
Let $A(G)$ and $D(G)$ be the adjacency matrix and the degree diagonal matrix of a graph $G$, respectively. Then $L(G)=D(G)-A(G)$ is called Laplacian matrix of the graph $G$. Let $G$ be a graph with $n$ vertices and $m$ edges. Then the $LI$-matrix of $G$ are defined as
$LI(G)=L(G)-\frac{2m}{n}I_n$,
where $I_n$ is the identity matrix. In this paper, we are interested in extremal properties of the Ky Fan $k$-norm of the $LI$-matrix of graphs, which is closely related to the well known problems and results in spectral graph theory, such as the Laplacian spectral radius, the Laplacian spread, the sum of the $k$ largest Laplacian eigenvalues, the Laplacian energy, and other parameters. Some bounds on the Ky Fan $k$-norm of the $LI$-matrix of graphs are given, and the extremal graphs are partly characterized. In addition, upper and lower bounds on the Ky Fan $k$-norm of $LI$-matrix of trees, unicyclic graphs and bicyclic graphs are determined, and the corresponding extremal graphs are characterized.

\bigskip


\end{abstract}
\baselineskip 20pt

\section{\large Introduction}
Let $G$ be a simple finite undirected graph with vertex set $V(G)$ and edge set $E(G)$.
The matrix $L(G)=D(G)-A(G)$ is called the Laplacian matrix of $G$, and its eigenvalues can be arranged as:
$$n\geq \mu_1(G)\geq \mu_2(G)\geq \cdots \geq \mu_{n-1}(G)\geq \mu_n(G)=0,$$
where $\mu_1(G)$ and $\mu_{n-1}(G)$ are called Laplacian spectral radius and algebraic connectivity of $G$, respectively. The investigation on the eigenvalues of Laplacian matrix of graphs is a topic of interest in spectral graph theory. There are amount of results on the eigenvalues of $L(G)$ in the literature, such as the Laplacian spectral radius \cite{GLS, PS}, the Laplacian spread $spr(L(G))$ \cite{AGRR, EK, ZSH}, the sum of the $k$ largest Laplacian eigenvalues $S_k(L(G))$ \cite{FHRT, GPRT, HMT}, the Laplacian energy $LE(G)$ \cite{DG, DMG, GZ}, etc. The $spr(L(G))$, $S_k(L(G))$ and $LE(G)$ are defined as follows:
$$spr(L(G))=\mu_1(G)-\mu_{n-1}(G), \quad S_k(L(G))=\sum\limits_{i=1}^{k}\mu_i(G), \quad LE(G)=\sum\limits_{i=1}^{n}\left\lvert \mu_i(G)-\frac{2m}{n} \right\rvert.$$

Recently, the trace norm of the adjacency matrix $A(G)$ of a graph $G$, defined as the sum of the singular values of $A(G)$, has been extensively studied under the name of graph energy \cite{LSG}. For generalizing and enriching the study of graph energy, Nikiforov \cite{N, N3} investigated the Ky Fan $k$-norm of adjacency matrix of a graph $G$, that is
$$||A(G)||_{F_k}=\sigma_1(A(G))+\sigma_2(A(G))+\cdots+\sigma_k(A(G)),$$
where $\sigma_1(A(G))\geq \sigma_2(A(G))\geq \cdots\geq \sigma_n(A(G))$ are the singular values of
the adjacency matrix $A(G)$, i.e. the nonnegative square roots of the eigenvalues of $A(G)A^T(G)$. Since the singular values of a real symmetric matrix are the moduli of its eigenvalues, the Ky Fan $k$-norm of adjacency matrix of a graph $G$ is also the sum of the $k$ largest absolute values of the eigenvalues of $A(G)$. He showed that some well-known problems and results in spectral graph theory are best stated in terms of the Ky Fan $k$-norm, for example, this norm is related to energy, spread, spectral radius, and other parameters. Thus he suggested to study arbitrary Ky Fan $k$-norm of graphs and proposed many interesting questions, especially the maximal Ky Fan
$k$-norm of graphs of given order. Later, Nikiforov has done a series of systematic in-depth analyses and researches for the Ky Fan
$k$-norm, which are not restricted to the adjacency matrix of graphs. One may refer to \cite{N, N1, N2, N3, NY} for more details on the Ky Fan $k$-norm.

Motivated by the above works, we study the Ky Fan $k$-norm of the $LI$-matrix of graphs. Let $G$ be a graph with $n$ vertices and $m$ edges. Then the $LI$-matrix of $G$ is defined as
$$LI(G)=L(G)-\frac{2m}{n}I_n.$$
 By the definition of the Ky Fan $k$-norm, we have
$||LI(G)||_{F_k}=\sum_{i=1}^{k}\sigma_i(LI(G))$,
where the singular values of $LI(G)$ is always indexed in decreasing order.
Clearly, $||LI(G)||_{F_n}=LE(G)$. Thus, a close examination of $||LI(G)||_{F_k}$ further advances the study of Laplacian energy of graphs.
In particular, $\sigma_1(LI(G))$ is called the spectral norm of the $LI$-matrix. Moreover, if $G$ is a regular graph, then $||LI(G)||_{F_k}=||A(G)||_{F_k}$.
From a geometric perspective, the Ky Fan $k$-norm of the $LI$-matrix of graphs represents the ordered sum of the distance between Laplacian eigenvalues and the average of all Laplacian eigenvalues, which is relevant to the hard problem that distribution of Laplacian eigenvalues of graphs in spectral graph theory, see \cite{JOT}.

In this paper, the extremal properties of the Ky Fan $k$-norm of the $LI$-matrix of graphs are studied. Around the following Nikiforov's question, upper and lower bounds on the Ky Fan $k$-norm of $LI$-matrix of trees, unicyclic graphs and bicyclic graphs are given, and the corresponding extremal graphs are characterized, which integrates previous results on the Laplacian spectral radius, the Laplacian spread and the sum of the $k$ largest Laplacian eigenvalues of trees, unicyclic graphs and bicyclic graphs.

\begin{question}{\bf (\cite{N})}
Study the extrema of the Ky Fan $k$-norm of a graph $G$, and their relations to the structure of $G$.
\end{question}

The rest of the paper is organized as follows. In Section 2, we introduce some notions and lemmas which we need to use in the proofs of our results. In Section 3, some properties on $\sigma_1(LI(G))$, $\sigma_n(LI(G))$ and $||LI(G)||_{F_2}$ of a graph $G$ are obtained. In Section 4, some bounds on $||LI(G)||_{F_k}$ of a graph $G$ are presented, and the extremal graphs are partly characterized. In Sections 5, lower and upper bounds on the Ky Fan $k$-norm of $LI$-matrix of trees, unicyclic graphs and bicyclic graphs are obtained, and the corresponding extremal graphs are characterized.

\section{\large  Preliminaries}

Denote by $K_n$, $P_n$, $C_n$ and $K_{1,\, n-1}$ the complete graph, path, cycle and star with $n$ vertices, respectively. For $v_i\in V(G)$, $d_G(v_i)=d_i(G)$ denotes the degree of vertex $v_i$ in $G$. The minimum and the maximum degree of $G$ are denoted by $\delta=\delta(G)$ and $\Delta=\Delta(G)$, respectively. We assume that $d_1(G)\geq d_2(G)\geq \cdots \geq d_n(G)$ and say that $d=(d_1(G), d_2(G), \ldots, d_n(G))$ is the degree sequence of the graph $G$. The conjugate of a degree sequence $d$ is the sequence $d^{*}=(d_1^*(G), d_2^*(G), \ldots, d_n^*(G))$ where $d_i^*(G)=|\{j: d_j(G)\geq i\}|$ is the number of vertices
of $G$ of degree at least $i$. For a graph $G$, the first Zagreb index $Z_1=Z_1(G)$ is defined as the sum of the squares of the vertices degrees.

A threshold graph may be obtained through an iterative process which starts with an isolated vertex, and at each step, either a new isolated vertex is added, or a vertex
adjacent to all previous vertices (dominating vertex) is added. The double star $S_{a,\, b}$ is the tree obtained from $K_2$ by attaching $a$ pendant edges to a vertex and $b$ pendant edges to the other. A connected graph is called a $c$-cyclic graph if it contains $n$ vertices and $n+c-1$ edges. Specially, if $c=0$, $1$ or $2$, then $G$ is called a tree, a unicyclic graph, or a bicyclic graph, respectively. The $G_{m,\,n}$, shown in Fig. 2.1, is a graph with $n$ vertices and $m$ edges which has $m-n+1$ triangles with a common edge and $2n-m+3$ pendent edges incident with one end vertex of the common edge. The join graph $G_1\vee G_2$ is the graph obtained from $G_1\cup G_2$ by joining every vertex of $G_1$ with every vertex of $G_2$.

\begin{lemma}{\bf (\cite{GM, K})}\label{le2,1} 
Let $G$ be a graph with $n$ vertices and at least one edge. Then $\Delta+1\leq \mu_1(G)\leq n$. The left equality for
connected graph holds if and only if $\Delta=n-1$, and the right equality holds if and only if the complement of $G$ is disconnected.
\end{lemma}

\begin{lemma}{\bf (\cite{T})}\label{le2,2} 
Let $A$ be an $m\times n$ matrix with singular values $\alpha_1\geq \alpha_2\geq \cdots \geq \alpha_{\min\{m,\,n\}}$. Let $B$ be a $p\times q$ submatrix of $A$, with singular values $\beta_1\geq \beta_2\geq \cdots \geq \beta_{\min\{p,\,q\}}$. Then
$\alpha_i\geq \beta_i\geq \alpha_{i+(m-p)+(n-q)}$
for $i=1, 2, \ldots, \min\{p,\,q\}$ and $i\leq \min\{p+q-m, p+q-n\}$.
\end{lemma}

\begin{lemma}{\bf (\cite{BH})}\label{le2,3} 
Let $G$ be a graph and let $H$ be a (not necessarily induced) subgraph of $G$ with $p$ vertices. Then $\mu_i(G)\geq \mu_i(H)$ for $1\leq i\leq p$ .
\end{lemma}

\begin{lemma}{\bf (\cite{LP})}\label{le2,4} 
Let $G$ be a connected graph on $n\geq 3$ vertices, with vertex degrees $d_1(G)\geq d_2(G)\geq \cdots \geq d_n(G)$. Then $\mu_2(G)\geq d_2(G)$.
\end{lemma}

\begin{lemma}{\bf (\cite{M1})}\label{le2,5} 
Let $G$ be a threshold graph on $n$ vertices with conjugate degree sequence $d^{*}=(d_1^*(G), d_2^*(G), \ldots, d_n^*(G))$. Then the Laplacian eigenvalue $\mu_i(G)=d_i^*(G)=d_i(G)+1$, $1\leq i \leq n-1$.
\end{lemma}

\begin{lemma}{\bf (\cite{YLT})}\label{le2,6} 
Let $\mathcal{T}_n$ the set of trees on $n$ vertices. Then
$$\mu_1(T_n)<\mu_1(S_{3,\, n-5})<\mu_1(T_n^4)<\mu_1(T_n^3)<\mu_1(S_{2,\, n-4})<\mu_1(S_{1,\, n-3})<\mu_1(K_{1,\,n-1})$$
for $T_n\in \mathcal{T}_n\setminus \{K_{1,\,n-1}, S_{1,\, n-3}, S_{2,\, n-4}, T_n^3, T_n^4, S_{3,\, n-5} \}$ and $T_n^i$ $(i=3, 4)$ shown in Fig. 2.1, where $\mu_1(S_{1,\, n-3})$, $\mu_1(S_{2,\, n-4})$, respectively, are the largest root of the following equations:
\begin{eqnarray*}
x^3-(n+2)x^2+(3n-2)x-n & = & 0,\\
x^3-(n+2)x^2+(4n-7)x-n & = & 0.
\end{eqnarray*}
\end{lemma}

\begin{lemma}{\bf (\cite{GZW})}\label{le2,7} 
For any tree $T_n$ with $n\geq 4$ vertices, $S_2(L(T_n))\leq S_2(L(S_{\lceil\frac{n-2}{2}\rceil, \, \lfloor\frac{n-2}{2}\rfloor}))$. The equality holds
if and only if $T_n\cong S_{\lceil\frac{n-2}{2}\rceil, \, \lfloor\frac{n-2}{2}\rfloor}$.
\end{lemma}

\begin{lemma}{\bf (\cite{LHS})}\label{le2,8} 
Let $G$ be a connected graph with $n\geq 12$. Then $\mu_1(G)+\mu_2(G)\geq 4+2\left(\cos\frac{\pi}{n}+\cos\frac{2\pi}{n}\right)$ with equality holding if and only if $G\cong P_n$.
\end{lemma}

\begin{lemma}{\bf (\cite{FXWL})}\label{le2,9} 
Let $T_n$ be a tree with $n\geq 5$ vertices. Then
$$spr(L(P_n))\leq spr(L(T_n))\leq spr(L(K_{1,\, n-1})).$$
The equality in the left hand side holds if and only if $T_n\cong P_n$, and the equality in the right hand side holds if and only if $T_n\cong K_{1,n-1}$.
\end{lemma}

\begin{lemma}{\bf (\cite{G, ZCL})}\label{le2,10} 
Let $\mathcal{U}_n$ the set of unicyclic graphs on $n$ vertices. Then
$$\mu_1(U_n)< \mu_1(G_{n,\,n}) \quad and \quad S_2(U_n)< S_2(G_{n,\,n})$$
for $U_n\in \mathcal{U}_n\setminus \{G_{n,\,n}\}$, $G_{n,\,n}$ shown in Fig. 2.1.
\end{lemma}

\begin{lemma}{\bf (\cite{BTF, YL})}\label{le2,11} 
Let $U_n$ be a unicyclic graph with $n\geq 4$. Then
$$spr(L(C_n))\leq spr(L(U_n))\leq spr(L(G_{n,\, n})).$$
The equality in the left hand side holds if and only if $U_n\cong C_n$, and the equality in the right hand side holds if and only if $U_n\cong G_{n,\, n}$.
\end{lemma}

\begin{lemma}{\bf (\cite{FLT, HSH, ZCLR})}\label{le2,12} 
Let $\mathcal{B}_n$ the set of bicyclic graphs on $n$ vertices. Then

{\normalfont (i)} $\mu_1(B_n)< \mu_1(B_n^{*})$ for $B_n\in \mathcal{B}_n\setminus \{B_n^{*}\}$, where $B_n^{*}$ is obtained from a star of order $n$ by adding two edges.

{\normalfont (ii)} $spr(L(B_n))<  spr(L(B_n^{*}))$ for $B_n\in \mathcal{B}_n\setminus \{B_n^{*}\}$.

{\normalfont (iii)} $S_2(L(B_n))< S_2(L(G_{n+1,\,n}))$
for $B_n\in \mathcal{B}_n\setminus \{G_{n+1,\,n}\}$, $G_{n+1,\,n}$ shown in Fig. 2.1.
\end{lemma}

\begin{picture}(300,80)
    \put(75,40){\circle*{3}}
    \put(75,40){\line(-5,4){23}}
    \put(75,40){\line(-5,-4){23}}
    \put(75,40){\circle*{3}}
    \put(75,40){\line(5,4){23}}
    \put(75,40){\line(5,-4){23}}
    \put(97,58){\circle*{3}}
    \put(97,58){\line(1,0){30}}
    \put(97,22){\circle*{3}}
    \put(97,22){\line(1,0){30}}
    \put(127,22){\circle*{3}}
    \put(127,58){\circle*{3}}
    \put(53,58){\circle*{3}}
    \put(53,35){\vdots}
    \put(53,22){\circle*{3}}
    \put(68,5){\footnotesize $T_n^3$}

    \put(199,40){\circle*{3}}
    \put(199,40){\line(-5,4){23}}
    \put(199,40){\line(-5,-4){23}}
    \put(199,40){\line(1,0){30}}
    \put(229,40){\circle*{3}}
    \put(229,40){\line(1,0){30}}
    \put(259,40){\circle*{3}}
    \put(259,40){\line(1,0){30}}
    \put(289,40){\circle*{3}}
    \put(177,58){\circle*{3}}
    \put(177,35){\vdots}
    \put(177,22){\circle*{3}}
    \put(242,5){\footnotesize $T_n^4$ }

    \put(351,40){\circle*{3}}
    \put(351,40){\line(-5,4){23}}
    \put(351,40){\line(-5,-4){23}}
    \put(351,40){\line(1,0){40}}
  \put(351,40){\line(1,0){40}}
   \put(371,70){\circle*{3}}
     \put(371,70){\line(-2,-3){20}}
     \put(371,70){\line(2,-3){20}}
      \put(371,50){\circle*{3}}
      \put(371,50){\line(-2,-1){20}}
     \put(371,50){\line(2,-1){20}}
    \put(391,40){\circle*{3}}
    \put(369,55){\vdots}
   \put(329,73){\line(2,-3){22}}
   \put(329,73){\circle*{3}}
    \put(329,58){\circle*{3}}
    \put(329,35){\vdots}
    \put(329,22){\circle*{3}}
\put(358,5){\footnotesize $G_{m\,n}$ }

\put(135,-16){Fig. 2.1 \quad Graphs $T_n^3$,  $T_n^4$ and $G_{m\,n}$. }
\end{picture}

\section{\large  Some properties on $\sigma_1(LI)$, $\sigma_n(LI)$  and $||LI||_{F_2}$ of a graph}

\begin{theorem}\label{th3,1} 
Let $G$ be a graph with $n$ vertices and $m\geq 1$ edges.

{\normalfont (i)} If $m\geq \frac{n^2}{4}$, then $\sigma_1(LI(G))=\frac{2m}{n}$.

{\normalfont (ii)} If $\Delta\geq \frac{4m}{n}-1$, then $\sigma_1(LI(G))=\mu_1(G)-\frac{2m}{n}$.

{\normalfont (iii)} If $G$ is a connected $r$-regular graph, then $\sigma_1(LI(G))=\frac{2m}{n}$.

{\normalfont (iv)} If $G$ is a bipartite graph, then $\sigma_1(LI(G))=\mu_1(G)-\frac{2m}{n}$.

\end{theorem}

\begin{proof} (i) If $m\geq \frac{n^2}{4}$, by Lemma \ref{le2,1}, we have $\mu_1(G)\leq n\leq \frac{4m}{n}$. Thus
$$\sigma_1(LI(G))=\max\left\{\mu_1(G)-\frac{2m}{n},\frac{2m}{n}\right\}=\frac{2m}{n}.$$

(ii) If $\Delta\geq \frac{4m}{n}-1$, by Lemma \ref{le2,1}, we have $\mu_1(G)\geq \Delta+1\geq  \frac{4m}{n}$. Thus
$$\sigma_1(LI(G))=\max\left\{\mu_1(G)-\frac{2m}{n},\frac{2m}{n}\right\}=\mu_1(G)-\frac{2m}{n}.$$

(iii) If $G$ is a connected $r$-regular graph, then $\mu_1(G)=r-\lambda_n(G)$, where $\lambda_n(G)$ is the least eigenvalue of the adjacency matrix of $G$. If $\mu_1(G)=r-\lambda_n(G)>\frac{4m}{n}=2r$,
that is $|\lambda_n(G)|>r=\lambda_1(G)$, a contradiction. Thus $\mu_1(G)\leq \frac{4m}{n}$. Further, $\sigma_1(LI(G))=\frac{2m}{n}$.

(iv) It is well known that the spectra of Laplacian matrix and signless Laplacian matrix coincide if and only if the graph $G$ is bipartite. From Lemma 2.1 in \cite{NLL}, we have $\mu_1(G)\geq \frac{4m}{n}$. Thus $\sigma_1(LI(G))=\mu_1(G)-\frac{2m}{n}$.

This completes the proof. $\Box$
\end{proof}

\begin{corollary}\label{cor3,1} 
Let $G$ be a graph with $n$ vertices and $m$ edges. If $\mu_1(G)<\frac{4m}{n}$, then $G$ contains odd cycles.
\end{corollary}

\begin{remark}
There exist non-bipartite graphs for which the equality $\sigma_1(LI(G))=\mu_1(G)-\frac{2m}{n}$ holds for $m<\frac{n^2}{4}$. The odd unicyclic graphs except odd cycle or bicyclic graphs with $n\geq 17$ vertices are examples. However, there also exist non-regular graphs for which the equality $\sigma_1(LI(G))=\frac{2m}{n}$ holds for $m<\frac{n^2}{4}$. The graph $H_n$ with $n$ vertices and $\frac{4}{3}n$ edges, depicted in Fig. 3.1, is an example. By direct calculations, we have $\mu_1(H_n)=5<\frac{16}{3}=\frac{4m}{n}$. Thus
$\sigma_1(LI(G))=\frac{2m}{n}=\frac{8}{3}$. It is interesting to characterize the graphs satisfying $\sigma_1(LI(G))=\frac{2m}{n}$ for $m<\frac{n^2}{4}$.

\begin{picture}(300,65)
    \put(35,40){\circle*{3}}
    \put(35,40){\line(-5,4){23}}
    \put(35,40){\line(-5,-4){23}}
    \put(35,40){\line(1,0){30}}
    \put(65,40){\circle*{3}}
    \put(65,40){\line(5,4){23}}
    \put(65,40){\line(5,-4){23}}
    \put(87,58){\circle*{3}}
    \put(87,58){\line(0,-1){36}}
    \put(87,22){\circle*{3}}
    \put(87,22){\line(2,3){24}}
   \put(111,58){\circle*{3}}
   \put(111,58){\line(2,-3){24}}
   \put(135,22){\circle*{3}}
   \put(135,22){\line(2,3){24}}
    \put(159,58){\circle*{3}}
    \put(159,58){\line(2,-3){24}}
    \put(159,58){\line(-2,0){48}}
    \put(183,22){\circle*{3}}
    \put(183,22){\line(2,3){24}}
    \put(183,22){\line(2,0){48}}
    \put(207,58){\circle*{3}}
    \put(207,58){\line(2,-3){24}}
    \put(231,22){\circle*{3}}
    \put(231,22){\line(2,3){24}}
    \put(255,58){\circle*{3}}
    \put(255,58){\line(2,-3){24}}
    \put(255,58){\line(2,0){48}}
    \put(279,22){\circle*{3}}
    \put(279,22){\line(2,3){24}}
    \put(303,58){\circle*{3}}
    \put(303,58){\line(2,-3){24}}
    \put(327,22){\circle*{3}}
    \put(327,22){\line(2,3){24}}
    \put(351,58){\circle*{3}}
    \put(351,58){\line(2,-3){24}}
    \put(375,22){\circle*{3}}
    \put(327,22){\line(2,0){48}}
    \put(400,22){\ldots}
    \put(375,22){\line(2,3){24}}
    \put(399,58){\circle*{3}}
    \put(13,58){\circle*{3}}
    \put(13,58){\line(0,-1){36}}
    \put(13,58){\line(1,0){74}}
    \put(13,22){\circle*{3}}

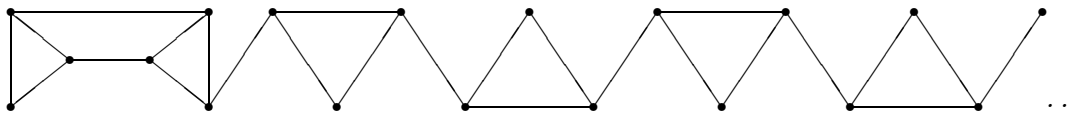
\put(40,-10){Fig. 3.1 \quad The graph $H_n$ ($n=3t$, $t\geq 2$ is the number of triangles). }
\end{picture}

\end{remark}

\begin{question}\label{qu3,1} 
Characterize all graphs $G$ satisfying $\sigma_1(LI(G))=\frac{2m}{n}$ for $m<\frac{n^2}{4}$.
\end{question}

\begin{theorem}\label{th3,2} 
Let $G$ be a connected graph on $n\geq 3$ vertices. Then
$$\mu_1(G)+\mu_2(G)> \frac{4m}{n}+1.$$
\end{theorem}

\begin{proof} Let $(d_1(G), d_2(G), \ldots, d_n(G))$ be the degree sequence of $G$. If $G$ is not a $r$-regular graph, by Lemmas \ref{le2,1} and \ref{le2,4}, we have
$$\mu_1(G)+\mu_2(G)\geq d_1(G)+ d_2(G)+1>\frac{2(d_1(G)+ d_2(G)+\cdots+d_n(G))}{n}+1=\frac{4m}{n}+1.$$

If $G=K_n$, then $\mu_1(G)+\mu_2(G)=2n>\frac{4m}{n}+1$. If $G$ is a $r$-regular graph without $K_{n}$, then
$\mu_1(G)+\mu_2(G)=2r-\lambda_{n-1}(G)-\lambda_{n}(G)$,
where $\lambda_{n-1}(G)$ and $\lambda_{n}(G)$ are the second least eigenvalue and the least eigenvalue of $A(G)$, respectively. Clearly, $\lambda_{n-1}(G)\leq 0$ and $\lambda_{n}(G)\leq -\frac{1+\sqrt{5}}{2}$ (see, e.g. \cite{H}). Thus
$$\mu_1(G)+\mu_2(G)=2r-\lambda_{n-1}(G)-\lambda_{n}(G)\geq 2r+\frac{1+\sqrt{5}}{2}>\frac{4m}{n}+1.$$

From the above arguments, we have the proof. $\Box$
\end{proof}

\begin{corollary}\label{cor3,2} 
Let $G$ be a connected graph on $n\geq 3$ vertices. If $\sigma_{n-1}(LI(G))>\sigma_n(LI(G))$, then $\sigma_n(LI(G))\neq \mu_1(G)-\frac{2m}{n}$.
\end{corollary}

\begin{proof} Suppose that $\sigma_n(LI(G))=\mu_1(G)-\frac{2m}{n}$. Then
$$\sigma_n(LI(G))=\mu_1(G)-\frac{2m}{n}< \sigma_{n-1}(LI(G))\leq \left\lvert \mu_2(G)-\frac{2m}{n}\right\lvert.$$

If $\mu_2(G)\geq \frac{2m}{n}$, then $\mu_1(G)<\mu_2(G)$, a contradiction.
If $\mu_2(G)<\frac{2m}{n}$, then $\mu_1(G)+\mu_2(G)< \frac{4m}{n}$, a contradiction.
Therefore $\sigma_n(LI(G))\neq \mu_1(G)-\frac{2m}{n}$. This completes the proof. $\Box$
\end{proof}

In \cite{GT}, Guo and Tan showed that $2$ is a Laplacian eigenvalue of any tree with perfect matchings. In \cite{AKM}, Akbari et al. provided a necessary and sufficient condition
under which a unicyclic graph with a perfect matching has Laplacian eigenvalue $2$. Thus there exists a tree $G$ or a unicyclic graph $G$ with $n$ vertices such that $\lim\limits_{n\rightarrow \infty}\sigma_{n}(LI(G))=0$.
Further, we propose the following question.

\begin{question}\label{qu3,2} 
Given a set $\mathcal{G}$ of graphs, whether there exists a graph $G\in \mathcal{G}$ such that $\lim\limits_{n\rightarrow \infty}\sigma_{n}(LI(G))=0$ or not.
\end{question}

\begin{theorem}\label{th3,3} 
Let $G$ be a graph with $n$ vertices and $m\geq 1$ edges. For any edge $uv\in E(G)$, we have
$$||LI(G)||_{F_2}\geq \left\lvert \frac{d_u+d_v+\sqrt{(d_u+d_v)^2+4}}{2}-\frac{2m}{n}\right\lvert+\left\lvert \frac{d_u+d_v-\sqrt{(d_u+d_v)^2+4}}{2}-\frac{2m}{n}\right\lvert.$$
\end{theorem}

\begin{proof}  Let $uv\in E(G)$. By Lemma \ref{le2,2}, we have $\sigma_1(LI(G))\geq \sigma_1'$ and $\sigma_2(LI(G))\geq \sigma_2'$, where $\sigma_1'$, $\sigma_2'$ are the singular values of the matrix
$$\begin{pmatrix}
d_u-\frac{2m}{n} & -1\\
-1 & d_v-\frac{2m}{n}
\end{pmatrix}.$$
Thus
\begin{eqnarray*}
||LI(G)||_{F_2} & = & \sigma_1(LI(G))+\sigma_2(LI(G))\\
& \geq & \sigma_1'+\sigma_2'\\
& = & \left\lvert \frac{d_u+d_v+\sqrt{(d_u+d_v)^2+4}}{2}-\frac{2m}{n}\right\lvert+\left\lvert \frac{d_u+d_v-\sqrt{(d_u+d_v)^2+4}}{2}-\frac{2m}{n}\right\lvert.
\end{eqnarray*}
 This completes the proof. $\Box$
\end{proof}

\begin{theorem}\label{th3,4} 
Let $G$ be a triangle-free graph with $n$ vertices and $m\geq 1$ edges. For any edge $uv\in E(G)$, we have
$$||LI(G)||_{F_2}\geq \sqrt{\Upsilon+2\sqrt{\Psi}},$$
where $\Upsilon=\left(d_u-\frac{2m}{n}\right)^2+\left(d_v-\frac{2m}{n}\right)^2+d_u+d_v$, $\Psi=\left(\left(d_u-\frac{2m}{n}\right)^2+d_u\right)\left(\left(d_v-\frac{2m}{n}\right)^2+d_v\right)-\left(d_u+d_v-\frac{4m}{n}\right)^2$.
\end{theorem}

\begin{proof} Let $uv\in E(G)$. By Lemma \ref{le2,2}, we have $\sigma_1(LI(G))\geq \sigma_1'$ and $\sigma_2(LI(G))\geq \sigma_2'$, where $\sigma_1'$, $\sigma_2'$ are the singular values of the matrix
$$B=\begin{pmatrix}
d_u-\frac{2m}{n} & -1 & * & \cdots & *\\
-1 & d_v-\frac{2m}{n} & * & \cdots & *
\end{pmatrix}_{2\times n}.$$
Since $G$ is a triangle-free graph, we have
$$BB^T=\begin{pmatrix}
(d_u-\frac{2m}{n})^2+d_u & \frac{4m}{n}-d_u-d_v\\
\frac{4m}{n}-d_u-d_v& (d_v-\frac{2m}{n})^2+d_v
\end{pmatrix}.$$
Thus the eigenvalues $x_1$ and $x_2$ of $BB^T$ are the roots of the following equations:
$$x^2-\left(\left(d_u-\frac{2m}{n}\right)^2+\left(d_v-\frac{2m}{n}\right)^2+d_u+d_v\right)x+~~~~~~~~~~~~~~~~~~~~~~~~~~~$$
$$\left(\left(d_u-\frac{2m}{n}\right)^2+d_u\right)\left(\left(d_v-\frac{2m}{n}\right)^2+d_v\right)-\left(d_u+d_v-\frac{4m}{n}\right)^2=0.$$
By Lemma \ref{le2,2}, we have
\begin{eqnarray*}
||LI(G)||_{F_2} & = & \sigma_1(LI(G))+\sigma_2(LI(G))\\
& \geq & \sigma_1'+\sigma_2'\\
& = & \sqrt{x_1}+\sqrt{x_2}\\
& = & \sqrt{x_1+x_2+2\sqrt{x_1x_2}}\\
& = & \sqrt{\Upsilon+2\sqrt{\Psi}},
\end{eqnarray*}
where $\Upsilon=\left(d_u-\frac{2m}{n}\right)^2+\left(d_v-\frac{2m}{n}\right)^2+d_u+d_v$, $\Psi=\left(\left(d_u-\frac{2m}{n}\right)^2+d_u\right)\left(\left(d_v-\frac{2m}{n}\right)^2+d_v\right)-\left(d_u+d_v-\frac{4m}{n}\right)^2$.
 This completes the proof. $\Box$
\end{proof}

\section{\large  Bounds on  $||LI||_{F_k}$  of a graph}

\begin{theorem}\label{th4,1} 
Let $G$ be a connected graph on $n\geq 3$ vertices and $m$ edges. If at most two of the $k$ vertices with the degree greater than or equal to $\frac{4m}{n}$ are  adjacent, then
$$||LI(G)||_{F_k}=S_k(L(G))-\frac{2km}{n}.$$
\end{theorem}

\begin{proof} Let $(d_1(G), d_2(G), \ldots, d_n(G))$ be the degree sequence of $G$. If $k$ vertices with the degree greater than or equal to $\frac{4m}{n}-1$ are not adjacent, then the star $K_{1, s}$ centered on $k$ vertices are subgraphs of $G$, where $s\geq \frac{4m}{n}-1$. By Lemma \ref{le2,3}, we have $\mu_k(G)\geq \mu_1(K_{1,\, s})= \frac{4m}{n}$. Thus we have
$$||LI(G)||_{F_k} = \sum\limits_{i=1}^{k}\sigma_i(LI(G))=\sum\limits_{i=1}^{k}\left\lvert \mu_k(G)-\frac{2m}{n}\right\lvert=S_k(L(G))-\frac{2km}{n}.$$

If two of the $k$ vertices with the degree greater than or equal to $\frac{4m}{n}$ are adjacent, then the star $K_{1, s}$ or the double star $S_{a,\,b}$ are subgraphs of $G$, where $s\geq \frac{4m}{n}-1$ and $a, b \geq \frac{4m}{n}-1$. By Lemmas \ref{le2,3} and \ref{le2,4}, we have $\mu_k(G)\geq \mu_1(K_{1,\, s})= \frac{4m}{n}$ or $\mu_k(G)\geq \mu_2(S_{a,\,b})\geq \frac{4m}{n}$. Thus we have
$$||LI(G)||_{F_k} = \sum\limits_{i=1}^{k}\sigma_i(LI(G))=\sum\limits_{i=1}^{k}\left\lvert \mu_k(G)-\frac{2m}{n}\right\lvert=S_k(L(G))-\frac{2km}{n}.$$

Combining the above arguments, we have the proof. $\Box$
\end{proof}

\begin{theorem}\label{th4,2} 
Let $G$ be a graph on $n$ vertices and $m$ edges, with vertex degrees $d_1(G)\geq d_2(G)\geq \cdots \geq d_n(G)$. If the subgraph $H$ of $G$ is a threshold graph with $d_k(H)\geq \frac{4m}{n}-1$, then
$$||LI(G)||_{F_k}=S_k(L(G))-\frac{2km}{n}.$$
\end{theorem}

\begin{proof} If $d_k(H)\geq \frac{4m}{n}-1$, by Lemmas \ref{le2,3} and \ref{le2,5}, then $\mu_k(G)\geq \mu_k(H)\geq \frac{4m}{n}$. Thus we have
$$||LI(G)||_{F_k} = \sum\limits_{i=1}^{k}\sigma_i(LI(G))=\sum\limits_{i=1}^{k}\left\lvert \mu_k(G)-\frac{2m}{n}\right\lvert=S_k(L(G))-\frac{2km}{n}.$$
This completes the proof. $\Box$
\end{proof}

\begin{theorem}\label{th4,3} 
Let $G$ be a graph with $n$ vertices and $m$ edges. Then
$$||LI(G)||_{F_k}\leq S_{k}(L(\overline{G}))+\frac{2m}{n}+(k-1)(n-\frac{2m}{n}).$$
The equality holds if $G=K_n$.
\end{theorem}

\begin{proof} Let $J_n$ be the all ones matrix of size $n$. Then $-L(\overline{G})=LI(G)+W_n$ and $W_n=(\frac{2m}{n}-n)I_n+J_n$.
Hence, the triangle inequality implies that
$$ ||LI(G)||_{F_k}\leq  ||L(\overline{G})||_{F_k}+||W_n||_{F_k},$$
that is,
$$||LI(G)||_{F_k}\leq S_{k}(L(\overline{G}))+\frac{2m}{n}+(k-1)(n-\frac{2m}{n}).$$
This completes the proof. $\Box$
\end{proof}

\begin{theorem}\label{th4,4} 
Let $G\neq K_n$ be a graph on $n$ vertices and $m\geq 1$ edges. If $m\geq \frac{n^2}{4}$ and $k\geq 2$, then
$$||LI(G)||_{F_k}\leq \frac{2m}{n}+(k-1)spr(L(G))$$
with equality if and only if $G=G_1\vee (K_1\cup G_2)$, $|V(G_2)|=n-\frac{2m}{n}-1$, $\mu_{n-1}(G_1)\geq \frac{4m}{n}-n$ and $\sigma_2(LI(G))=\cdots=\sigma_k(LI(G))=\mu_1(G)-\frac{2m}{n}$.
\end{theorem}

\begin{proof} Since $m\geq \frac{n^2}{4}$, by Theorem \ref{th3,1}, we have $\sigma_1(LI(G))=\frac{2m}{n}$. Let $x+y=k-1$, $0\leq x, y\leq k-1$. By the Cauchy-Schwarz inequality, we have
\begin{eqnarray*}
||LI(G)||_{F_k} & = & \sum\limits_{i=1}^{k}\sigma_i(LI(G))\\
& \leq &\frac{2m}{n}+x(\mu_1(G)-\frac{2m}{n})+y(\frac{2m}{n}-\mu_{n-1}(G))\\
& \leq & \frac{2m}{n} +\sqrt{x^2+y^2}\sqrt{(\mu_1(G)-\frac{2m}{n})^2+(\frac{2m}{n}-\mu_{n-1}(G))^2}\\
& \leq & \frac{2m}{n}+(k-1)spr(L(G))
\end{eqnarray*}
with equality if and only if $\sigma_2(LI(G))=\cdots=\sigma_k(LI(G))=\mu_1(G)-\frac{2m}{n}$, $x=k-1$ and $\mu_{n-1}(G)=\frac{2m}{n}$. From Theorem 1 in \cite{LL}, we have $\mu_{n-1}(G)=\frac{2m}{n}$ if and only if $G=G_1\vee (K_1\cup G_2)$, $|V(G_2)|=n-\frac{2m}{n}-1$ and $\mu_{n-1}(G_1)\geq \frac{4m}{n}-n$.
This completes the proof. $\Box$
\end{proof}

\begin{remark}
If $G=C_4\vee (K_1\cup K_1)$, then
$||LI(G)||_{F_k}= \frac{2m}{n}+(k-1)spr(L(G))$
for $k=1, 2, 3$.
\end{remark}

\begin{theorem}\label{th4,5} 
Let $G$ be a graph with $n$ vertices and $m$ edges. If $k\geq 2$ and $m\geq \frac{n^2}{4}$, then
$$||LI(G)||_{F_k}\leq \frac{2m}{n}+\sqrt{(k-1)\left(2m+Z_1-\frac{4m^2}{n}-\frac{4m^2}{n^2}\right)}. \eqno{(4.2)} $$
The equality holds in $(4.2)$ if and only if $G$ is a graph satisfying
$$\left\{
\begin{array}{llll}
\sigma_2(LI(G))=\sigma_3(LI(G))=\cdots=\sigma_k(LI(G)) ,\\
\sigma_{k+1}(LI(G))=\sigma_{k+2}(LI(G))=\cdots=\sigma_n(LI(G))=0.\\
\end{array}
\right.
$$
\end{theorem}

\begin{proof} Since $m\geq \frac{n^2}{4}$, by Theorem \ref{th3,1}, we have $\sigma_1(LI(G))=\frac{2m}{n}$. Since $\sum_{i=1}^{n}\sigma_i^2(LI(G))=2m+Z_1-\frac{4m^2}{n}$,
by the Cauchy-Schwarz inequality, we have
\begin{eqnarray*}
||LI(G)||_{F_k} & = & \sum\limits_{i=1}^{k}\sigma_i(LI(G))\\
& \leq &\frac{2m}{n}+\sqrt{(k-1)\sum\limits_{i=2}^{k}\sigma_i^2(LI(G))}\\
& \leq & \frac{2m}{n}+\sqrt{(k-1)\sum\limits_{i=2}^{n}\sigma_i^2(LI(G))}\\
& = & \frac{2m}{n}+\sqrt{(k-1)\left(2m+Z_1-\frac{4m^2}{n}-\frac{4m^2}{n^2}\right)}.
\end{eqnarray*}
Hence the equality holds in (4.2) if and only if $G$ is a graph satisfying
$$\left\{
\begin{array}{llll}
\sigma_2(LI(G))=\sigma_3(LI(G))=\cdots=\sigma_k(LI(G)) ,\\
\sigma_{k+1}(LI(G))=\sigma_{k+2}(LI(G))=\cdots=\sigma_n(LI(G))=0.\\
\end{array}
\right.
$$
The proof is completed. $\Box$
\end{proof}

\begin{remark}
For a fixed $k$, we can use complete regular $r$-partite graphs or threshold graphs to construct graphs such that the equality holds in (4.2). However, it is an open problem to find all the graphs such that the equality holds in (4.2).
\end{remark}

\begin{theorem}\label{th4,6} 
Let $G$ be a graph with $n$ vertices and $m>1$ edges, and let $d_2$ be the second largest degree of $G$. If $d_2\geq \frac{4m}{n}$, then
$$||LI(G)||_{F_k}\leq kn-\frac{2km}{n}.\eqno{(4.3)}$$
\end{theorem}

\begin{proof} Since $d_2\geq \frac{4m}{n}$, by Lemma \ref{le2,4}, we have $|\mu_2-\frac{2m}{n}|\geq |\mu_i-\frac{2m}{n}|$ for $i=3, 4, \ldots, n$.
By Lemma \ref{le2,1}, we have
\begin{eqnarray*}
||LI(G)||_{F_k} & = & \sum\limits_{i=1}^{k}\sigma_i(LI(G))\\
& \leq & \mu_1(G)-\frac{2m}{n}+(k-1)\left(\mu_2(G)-\frac{2m}{n}\right)\\
& = & \mu_1(G)+(k-1)\mu_2(G)-\frac{2km}{n}\\
& \leq & kn-\frac{2km}{n}.
\end{eqnarray*}
This completes the proof. $\Box$
\end{proof}

\begin{remark}
If $G$ is a complete split graph $K_k\vee (n-k)K_1$ with $k\leq \lfloor\frac{2n-1-\sqrt{2n^2-2n+1}}{2}\rfloor$, then the equality in $(4.3)$ holds.
\end{remark}

Let $x=(x_1, x_2, \ldots, x_n)$ and $y=(y_1, y_2, \ldots, y_n)$ be two
non-increasing sequences of real numbers. If
$\sum_{i=1}^{j}x_i\leq \sum_{i=1}^{j}y_i$ for $j=1, 2, \ldots, n$,
then we say that $x$ is weakly majorized by $y$ and denote $x\prec_{w} y$. If in addition
to $x\prec_{w} y$, $\sum_{i=1}^{n}x_i=\sum_{i=1}^{n}y_i$ holds, then we say that $x$ is majorized by $y$
and denote $x\prec y$. From \cite{MO}, if $f(t)$ is a convex function, then $x\prec y$ implies $(f(x_1), f(x_2), \ldots, f(x_n)) \prec_{w} (f(y_1), f(y_2), \ldots, f(y_n))$.

\begin{theorem}\label{th4,7} 
Let $G$ be a graph with $n$ vertices and $m$ edges. If $ m\geq\frac{n^2}{4}$ and there is $\alpha$ such that $\mu_1(G)\geq \alpha\geq \frac{2m}{n-1}$, then
$$||LI(G)||_{F_k} \geq \frac{2m}{n}+\left\lvert \alpha-\frac{2m}{n}\right\lvert+(k-2)\left\lvert\frac{2m-\alpha}{n-2}-\frac{2m}{n}\right\lvert$$
for $k\geq 2$.
\end{theorem}

\begin{proof} Let $x=(\alpha, \frac{2m-\alpha}{n-2}, \ldots, \frac{2m-\alpha}{n-2},0)$, $y=(\mu_1(G), \mu_2(G), \ldots, \mu_{n-1}(G),0)\in \mathbb{R}^{n}$. Then $x\prec y$. By Theorem \ref{th3,1}, we have $\sigma_1(LI(G))= \frac{2m}{n}$ for $ m\geq\frac{n^2}{4}$. Since $f(t)=|t-\frac{2m}{n}|$ is the convex function, we have
\begin{eqnarray*}
||LI(G)||_{F_k} & = & \sigma_1(LI(G))+\sigma_2(LI(G))+\cdots+\sigma_k(LI(G))\\
& \geq & \frac{2m}{n}+\left\lvert \mu_1(G)-\frac{2m}{n}\right\lvert+\cdots+\left\lvert \mu_{k-1}(G)-\frac{2m}{n}\right\lvert\\
& \geq & \frac{2m}{n}+\left\lvert \alpha-\frac{2m}{n}\right\lvert+(k-2)\left\lvert\frac{2m-\alpha}{n-2}-\frac{2m}{n}\right\lvert
\end{eqnarray*}
for $k\geq 2$. This completes the proof. $\Box$
\end{proof}

\begin{corollary}\label{cor4,1} 
Let $G$ be a graph with $n$ vertices and $m$ edges. If $ m\geq\frac{n^2}{4}$, then
$$||LI(G)||_{F_k}\geq  \Delta+1+(k-2)\left\lvert\frac{2m-\Delta-1}{n-2}-\frac{2m}{n}\right\lvert,$$
for $k\geq 2$.
\end{corollary}

\begin{proof}
By Lemma \ref{le2,1} and Theorem \ref{th4,7}, we have the proof. $\Box$
\end{proof}

\section{\large On the $||LI||_{F_k}$ of $c$-cyclic graphs}

\begin{theorem}\label{th5,1} 
Let $T_n$ be a tree with $n$ vertices. Then
\begin{eqnarray*}
\sigma_1(LI(P_n))& < & \sigma_1(LI(T_n))<\sigma_1(LI(S_{3,\, n-5}))< \sigma_1(LI(T_n^4))<\sigma_1(LI(T_n^3))\\
& < & \sigma_1(LI(S_{2,\, n-4}))<\sigma_1(LI(S_{1,\, n-3}))<\sigma_1(LI(K_{1,\,n-1}))
\end{eqnarray*}
for $T_n\in \mathcal{T}_n\setminus \{K_{1,\,n-1}, S_{1,\, n-3}, S_{2,\, n-4}, T_n^3, T_n^4, S_{3,\, n-5}, P_n \}$.
\end{theorem}

\begin{proof} By Theorems \ref{th3,1}, we have $||LI(T_n)||_{F_1}=||QI(T_n)||_{F_1}=\mu_1(T_n)-2+\frac{2}{n}$.
It is well known that $\mu_1(P_n)\leq \mu_1(T_n)$ for all trees with $n$ vertices (see, e.g. \cite{PG}). By Lemma \ref{le2,6}, we have
\begin{eqnarray*}
\sigma_1(LI(P_n))& < & \sigma_1(LI(T_n))<\sigma_1(LI(S_{3,\, n-5}))< \sigma_1(LI(T_n^4))<\sigma_1(LI(T_n^3))\\
& < & \sigma_1(LI(S_{2,\, n-4}))<\sigma_1(LI(S_{1,\, n-3}))<\sigma_1(LI(K_{1,\,n-1}))
\end{eqnarray*}
for $T_n\in \mathcal{T}_n\setminus \{K_{1,\,n-1}, S_{1,\, n-3}, S_{2,\, n-4}, T_n^3, T_n^4, S_{3,\, n-5}, P_n \}$.
This completes the proof. $\Box$
  \end{proof}

\begin{theorem}\label{th5,2} 
Let $T_n$ be a tree with $n\geq 12$ vertices. Then
$$||LI(P_n)||_{F_2}< ||LI(T_n)||_{F_2}<||LI(S_{\lceil\frac{n-2}{2}\rceil, \, \lfloor\frac{n-2}{2}\rfloor})||_{F_2}<||LI(S_{1,\,n-3})||_{F_2}<||LI(K_{1,\,n-1})||_{F_2}$$
for $T_n\in \mathcal{T}_n\setminus \{K_{1,\,n-1}, S_{1,\, n-3}, S_{\lceil\frac{n-2}{2}\rceil, \, \lfloor\frac{n-2}{2}\rfloor}, P_n \}$.
\end{theorem}

\begin{proof} For the upper bounds, we consider two cases depending on $\sigma_2(LI(T_n))$.

{\bf Case 1.} $\sigma_2(LI(T_n))=\mu_2(T_n)-2+\frac{2}{n}$. Then $||LI(T_n)||_{F_2}=\mu_1(T_n)+\mu_2(T_n)-4+\frac{4}{n}$. By Lemma \ref{le2,7}, we have
$$||LI(T_n)||_{F_2}\leq S_2(L(S_{\lceil\frac{n-2}{2}\rceil, \, \lfloor\frac{n-2}{2}\rfloor}))-4+\frac{4}{n}.$$

{\bf Case 2.} $\sigma_2(LI(T_n))=2-\frac{2}{n}$. Then $||LI(T_n)||_{F_2}= \mu_1(T_n)$. Let $d_2(T_n)$ be the second largest degree of $T_n$. If $d_2(T_n)\geq 4$, by Lemma \ref{le2,4}, then $\left\lvert\mu_2(T_n)-\frac{2m}{n}\right\lvert\geq \left\lvert\mu_i(T_n)-\frac{2m}{n}\right\lvert$ for $i=3,4,\ldots, n$, that is, $\sigma_2(LI(T_n))=\mu_2(T_n)-2+\frac{2}{n}$, a contradiction. Thus $d_2(T_n)\leq 3$. By Lemma \ref{le2,6}, we have $$\mu_1(T_n)<\mu_1(S_{3,\, n-5})<\mu_1(T_n^4)<\mu_1(T_n^3)<\mu_1(S_{2,\, n-4})<\mu_1(S_{1,\, n-3})<\mu_1(K_{1,\,n-1})$$
for $T_n\in \mathcal{T}_n\setminus \{K_{1,\,n-1}, S_{1,\, n-3}, S_{2,\, n-4}, T_n^3, T_n^4, S_{3,\, n-5} \}$ and $T_n^i$ $(i=3, 4)$ shown in Fig. 2.1, where $\mu_1(S_{1,\, n-3})$, $\mu_1(S_{2,\, n-4})$, respectively, are the largest root of the following polynomials:
\begin{eqnarray*}
f_1(x) & = & x^3-(n+2)x^2+(3n-2)x-n,\\
f_2(x) & = & x^3-(n+2)x^2+(4n-7)x-n.
\end{eqnarray*}

By derivative, we know that $f_1'(x)>0$ for $x\in (n-1, +\infty)$. Therefore $f_1(x)$ is strictly increasing on $(n-1, +\infty)$.
Since $f_1(n-1)=-1<0$ and $f_1(n)=n(n-3)>0$ for $n\geq 10$, we have $n-1<\mu_1(S_{1,\,n-3})<n$.

By derivative, we know that $f'_2(x)>0$ for $x\in (n-2, +\infty)$. Therefore $f_2(x)$ is strictly increasing on $(n-2, +\infty)$.
Since $f_2(n-2)=-2<0$ and $f_2(n-2+\frac{1}{n})=\frac{1}{n^3}(n^4-10n^3+15n^2-8n+1)>0$ for $n\geq 10$, we have $n-2<\mu_1(S_{2,\,n-4})<n-2+\frac{1}{n}$.

By direct calculation, the Laplacian characteristic polynomial of $S_{\lceil\frac{n-2}{2}\rceil, \, \lfloor\frac{n-2}{2}\rfloor}$ is
$$\phi(x)=x(x-1)^{n-4}[x^3-(n+2)x^2+(2n+\lceil\frac{n-2}{2}\rceil \lfloor\frac{n-2}{2}\rfloor+1)x-n].$$
It is well known that  $\mu_{n-1}(T_n)\leq \delta(T_n)=1$ (see, e.g. \cite{F}).  By Lemma \ref{le2,3}, we have $\mu_{2}(S_{\lceil\frac{n-2}{2}\rceil, \, \lfloor\frac{n-2}{2}\rfloor})\geq 2$. Thus $\mu_1(S_{\lceil\frac{n-2}{2}\rceil, \, \lfloor\frac{n-2}{2}\rfloor})$, $\mu_{2}(S_{\lceil\frac{n-2}{2}\rceil, \, \lfloor\frac{n-2}{2}\rfloor})$ and $\mu_{n-1}(S_{\lceil\frac{n-2}{2}\rceil, \, \lfloor\frac{n-2}{2}\rfloor})$ are roots of the following polynomial
$$g(x)= x^3-(n+2)x^2+(2n+\lceil\frac{n-2}{2}\rceil \lfloor\frac{n-2}{2}\rfloor+1)x-n.$$
Since $\mu_{n-1}(G)=n-\mu_1(\overline{G})$ (see, e.g. \cite{M}), by the Vieta Theorem, we have
$$\mu_1(S_{\lceil\frac{n-2}{2}\rceil, \, \lfloor\frac{n-2}{2}\rfloor})+\mu_{2}(S_{\lceil\frac{n-2}{2}\rceil, \, \lfloor\frac{n-2}{2}\rfloor})=n+2- \mu_{n-1}(S_{\lceil\frac{n-2}{2}\rceil, \, \lfloor\frac{n-2}{2}\rfloor})=\mu_1(\overline{S_{\lceil\frac{n-2}{2}\rceil, \, \lfloor\frac{n-2}{2}\rfloor}})+2.$$
Thus
$$||LI(S_{\lceil\frac{n-2}{2}\rceil, \, \lfloor\frac{n-2}{2}\rfloor})||_{F_2}=\mu_1(\overline{S_{\lceil\frac{n-2}{2}\rceil, \, \lfloor\frac{n-2}{2}\rfloor}})-2+\frac{4}{n}<n-2+\frac{4}{n}<\mu_1(S_{1,\,n-3}).$$

If $n$ is an odd number, by direct calculation, we have that the Laplacian characteristic polynomial of $\overline{S_{\frac{n-1}{2}, \, \frac{n-3}{2}}}$ is
$$\varphi(x)=\frac{1}{4}x(x-n+1)^{n-4}[4x^3-8(n-1)x^2+(5n^2-12n+7)x-n^3+4n^2-3n].$$
It follows that $\mu_1(\overline{S_{\frac{n-1}{2}, \, \frac{n-3}{2}}})$ is the largest root of the following polynomial
$$h_1(x)=4x^3-8(n-1)x^2+(5n^2-12n+7)x-n^3+4n^2-3n.$$
Thus $||LI(S_{\frac{n-1}{2}, \, \frac{n-3}{2}})||_{F_2}$ is the largest root of the polynomial $h_1(x+2-\frac{4}{n})$. Noting that $f_2(x)$ and $h_1(x+2-\frac{4}{n})$ are strictly increasing on $(n-2, +\infty)$. Since
\begin{eqnarray*}
h_1(x+2-\frac{4}{n})-f_2(x) & = & 3x^3-(7n+\frac{48}{n}-34)x^2+(5n^2-48n-\frac{256}{n}\\
& & +\frac{192}{n^2}+158)x-n^3+14n^2-78n-\frac{476}{n}\\
& & +\frac{512}{n^2}-\frac{256}{n^3}+254\\
& < & 0
\end{eqnarray*}
for $x\in (n-2, n-2+1/n)$, we have $||LI(S_{\frac{n-1}{2}, \, \frac{n-3}{2}})||_{F_2}>\mu_1(S_{2,\,n-4})$.

If $n$ is an even number, by a similar reasoning as the above, we can conclude that
$||LI(S_{\frac{n-2}{2}, \, \frac{n-2}{2}})||_{F_2}>\mu_1(S_{2,\,n-4})$. Therefore, $||LI(S_{\lceil\frac{n-2}{2}\rceil, \, \lfloor\frac{n-2}{2}\rfloor})||_{F_2}>\mu_1(S_{2,\,n-4})$.

It is easy to see that $||LI(K_{1,\,n-1})||_{F_2}=\mu_1(K_{1,\,n-1})$, $||LI(S_{1,\,n-3})||_{F_2}=\mu_1(S_{1,\,n-3})$ and $||LI(S_{2,\,n-4})||_{F_2}=\mu_1(S_{2,\,n-4})$. Combining the above arguments, we have
$$||LI(T_n)||_{F_2}<||LI(S_{\lceil\frac{n-2}{2}\rceil, \, \lfloor\frac{n-2}{2}\rfloor})||_{F_2}<||LI(S_{1,\,n-3})||_{F_2}<||LI(K_{1,\,n-1})||_{F_2}$$
for $T_n\in \mathcal{T}_n\setminus \{K_{1,\,n-1}, S_{1,\, n-3}, S_{\lceil\frac{n-2}{2}\rceil, \, \lfloor\frac{n-2}{2}\rfloor}\}$.

Now we will show that $||LI(P_n)||_{F_2}< ||LI(T_n)||_{F_2}$ for $T_n\in \mathcal{T}_n\setminus \{P_n\}$. By Lemma \ref{le2,8}, we have $S_2(L(P_n))\leq S_2(L(T_n))$. Since $\mu_1(P_n)\leq \mu_1(T_n)$, we have
\begin{eqnarray*}
||LI(T_n)||_{F_2} & = & \sigma_1(LI(T_n))+\sigma_2(LI(T_n))\\
& = & \max\{\mu_1(T_n)+\mu_2(T_n)-4+\frac{4}{n}, \mu_1(T_n)\}\\
& = & \max\{S_2(L(T_n))-4+\frac{4}{n}, \mu_1(T_n)\}\\
& \geq & ||LI(P_n)||_{F_2}
\end{eqnarray*}
with equality if and only if $T_n\cong P_n$. This completes the proof. $\Box$
\end{proof}

\begin{theorem}\label{th5,3} 
Let $T_n$ be a tree with $n\geq 5$ vertices. Then
$$||LI(T_n)||_{F_3}\leq n+1-\frac{2}{n}$$
with equality if and only if $T_n\cong K_{1,\,n-1}$.
\end{theorem}

\begin{proof} From Theorem 1.1 in \cite{FHRT}, we have $S_2(L(T_n))< n+2-\frac{2}{n}$ and $S_3(L(T_n))< n+4-\frac{4}{n}$. By Lemma \ref{le2,9}, we have $spr(L(T_n))<spr(L(K_{1,\,n-1}))=n-1$ for $T_n\in \mathcal{T}_n\setminus \{K_{1,\,n-1}\}$. Thus
\begin{eqnarray*}
||LI(T_n)||_{F_3} & = & \sigma_1(LI(T_n))+\sigma_2(LI(T_n))+\sigma_3(LI(T_n))\\
& = &  \max\{S_2(L(T_n))-2+\frac{2}{n}, S_3(L(T_n))-6+\frac{6}{n}, spr(L(T_n))+2-\frac{2}{n}\}\\
& \leq & \max\{n, n-2+\frac{2}{n}, spr(L(T_n))+2-\frac{2}{n}\}\\
& \leq & ||LI(K_{1,\,n-1})||_{F_3}\\
& = & n+1-\frac{2}{n}
\end{eqnarray*}
with equality if and only if $T_n\cong K_{1,\,n-1}$.
This completes the proof. $\Box$
\end{proof}

\begin{conjecture}
Let $T_n$ be a tree with $n\geq 12$ vertices. Then
$$||LI(P_n)||_{F_k}\leq ||LI(T_n)||_{F_k}\leq ||LI(K_{1, \, n-1})||_{F_k}.$$
The equality in the left hand side holds if and only if $T_n\cong P_n$, and the equality in the right hand side holds if and only if $T_n\cong K_{1,\,n-1}$.
\end{conjecture}

\begin{theorem}\label{th5,4} 
Let $U_n$ be a unicyclic graph with $n\geq 5$ vertices. Then
$$\sigma_1(LI(C_n))\leq \sigma_1(LI(U_n))\leq n-2.$$
The equality in the left hand side holds if and only if $U_n\cong C_n$, and the equality in the right hand side holds if and only if $U_n\cong G_{n,\,n}$.
\end{theorem}

\begin{proof} If $\Delta(U_n)\geq 3$, by Lemma \ref{le2,1}, we have $\mu_1(U_n)\geq \Delta(U_n)+1\geq 4$. Thus $\sigma_1(LI(U_n))=\mu_1(U_n)-2$. By Lemma \ref{le2,10}, we have $\sigma_1(LI(U_n))\leq \sigma_1(LI(G_{n,\,n}))$ with equality if and only if $U_n\cong G_{n,\,n}$. From \cite{SSG}, it follows that $U_n^2$, shown in Fig. 5.1, is the smallest Laplacian spectral radii among all unicyclic graphs with $\Delta(U_n)\geq 3$. Hence $\sigma_1(LI(U_n))\geq \sigma_1(LI(U_n^2))$. By Lemma \ref{le2,3}, we have $\sigma_1(LI(U_n^2))=\mu_1(U_n^2)-2\geq \mu_1(U_5^2)-2> 2.17008$. Thus $\sigma_1(LI(U_n))> 2.17008$.

If $\Delta(U_n)=2$, then $U_n=C_n$. Thus $\sigma_1(LI(U_n))=\sigma_1(LI(C_n))=2$.

Combining the above arguments, we have the proof. $\Box$

\end{proof}

\begin{theorem}\label{th5,5} 
Let $U_n$ be a unicyclic graph with $n\geq 12$ vertices. Then
$$||LI(C_n)||_{F_2} \leq ||LI(U_n)||_{F_2}\leq n.$$
The equality in the left hand side holds if and only if $U_n\cong C_n$, and the equality in the right hand side holds if and only if $U_n\cong G_{n,\,n}$.
\end{theorem}

\begin{proof}
Since $||LI(U_n)||_{F_2}=\max\{\mu_1(U_n), S_2(L(U_n))-4 \}$, by Lemma \ref{le2,10}, we have $||LI(U_n)||_{F_2}\leq ||LI(G_{n,\,n})||_{F_2}$ with equality if and only if $U_n\cong G_{n,\,n}$. It is well known that $S_2(L(C_n))=6+2\cos\frac{2\pi}{n}$ for even cycle, and $S_2(L(C_n))=4+4\cos\frac{\pi}{n}$ for odd cycle. Thus $S_2(L(C_n))<8$. From the proof of Lemma 4.4 in \cite{LHS}, it follows that $\mu_1(T_n)+\mu_2(T_n)=q_1(T_n)+q_2(T_n)\geq 8$ for $n\geq 12$. By Lemma \ref{le2,3}, we have $S_2(L(U_n))=\mu_1(U_n)+\mu_2(U_n)\geq 8$ for $n\geq 12$. Hence $S_2(L(U_n))>S_2(L(C_n))$ for $\mathcal{U}_n\setminus \{C_n\}$. By the proof of Theorem \ref{th5,4}, we have $\mu_1(U_n)>\mu_1(C_n)$ for $\mathcal{U}_n\setminus \{C_n\}$. Therefore, $||LI(U_n)||_{F_2}=\max\{\mu_1(U_n), S_2(L(U_n))-4 \} \geq ||LI(C_n)||_{F_2}$ with equality if and only if $U_n\cong C_n$. This completes the proof. $\Box$
\end{proof}

\begin{theorem}\label{th5,6} 
Let $U_n$ be a unicyclic graph with $n\geq 12$ vertices. Then
$$ ||LI(U_n)||_{F_3}\leq n+1$$
with equality if and only if $U_n\cong G_{n,\,n}$.
\end{theorem}

\begin{proof} From Corollary 4.1 in \cite{DZ}, we have $S_3(L(U_n))\leq n+6$. By Lemmas \ref{le2,10} and \ref{le2,11}, we have $S_2(L(U_n))< S_2(L(G_{n,\,n}))$ and $spr(L(U_n))<spr(L(G_{n,\,n}))=n-1$ for
$U_n\in \mathcal{U}_n\setminus \{G_{n,\,n}\}$. Thus
\begin{eqnarray*}
||LI(U_n)||_{F_2} & = & \sigma_1(LI(U_n))+\sigma_2(LI(U_n))+\sigma_3(LI(U_n))\\
& = &  \max\{S_2(L(U_n))-2, S_3(L(U_n))-6, spr(L(U_n))+2\}\\
& \leq & ||LI(G_{n,\,n})||_{F_3}\\
& = & n+1
\end{eqnarray*}
with equality if and only if $U_n\cong G_{n,\,n}$.
This completes the proof. $\Box$
\end{proof}

\begin{theorem}\label{th5,7} 
Let $B_n$ be a bicyclic graph with $n\geq 17$ vertices. Then
$$\sigma_1(LI(B_n^1))\leq \sigma_1(LI(B_n))\leq \sigma_1(LI(B_n^{*})).$$
The equality in the left hand side holds if and only if $B_n\cong B_n^1$, and the equality in the right hand side holds if and only if $B_n\cong B_n^{*}$.
\end{theorem}

\begin{proof} Since $B_n^1=U_n^2+e$, by Lemma \ref{le2,3}, we have
$$\mu_1(B_n^1)-2-\frac{2}{n}\geq \mu_1(U_n^2)-2-\frac{2}{n}\geq \mu_1(U_{10}^2)-2-\frac{2}{n}> 2.23566-\frac{2}{n}>2+\frac{2}{n}$$
for $n\geq 17$. From \cite{SSG}, we know that $B_n^1$, shown in Fig. 5.1, is the smallest Laplacian spectral radii among all bicyclic graphs. Thus $\sigma_1(LI(B_n))=\mu_1(B_n)-2-\frac{2}{n}\geq \mu_1(B_n^1)-2-\frac{2}{n}$ with equality if and only if $G\cong B_n^1$. By Lemma \ref{le2,12}, we have
$\sigma_1(LI(B_n))=\mu_1(B_n)-2-\frac{2}{n}\leq \mu_1(B_n^{*})-2-\frac{2}{n}=\sigma_1(LI(B_n^{*}))$ with equality if and only if $B_n\cong B_n^{*}$.
This completes the proof. $\Box$
\end{proof}

\begin{theorem}\label{th5,8} 
Let $B_n$ be a bicyclic graph with $n$ vertices. Then
$$ ||LI(B_n)||_{F_2}\leq ||LI(G_{n+1,\,n})||_{F_2}$$
with  equality if and only if $B_n\cong G_{n+1,\,n}$.
\end{theorem}

\begin{proof}
Since $||LI(B_n)||_{F_2}=\max\{\mu_1(B_n), S_2(L(B_n))-4-\frac{4}{n} \}$, by Lemma \ref{le2,12}, we have the proof. $\Box$
\end{proof}

\begin{theorem}\label{th5,9} 
Let $B_n$ be a bicyclic graph with $n$ vertices. Then
$$ ||LI(B_n)||_{F_3}\leq n+2-\frac{2}{n}$$
with equality if and only if $B_n\cong G_{n+1,\,n}$.
\end{theorem}

\begin{proof} From Corollary 4.2 in \cite{DZ}, we have $S_3(L(B_n))\leq n+7$. By Lemma \ref{le2,12}, we have
\begin{eqnarray*}
||LI(B_n)||_{F_2} & = & \sigma_1(LI(U_n))+\sigma_2(LI(U_n))+\sigma_3(LI(U_n))\\
& = &  \max\{S_2(L(B_n))-2-\frac{2}{n}, S_3(L(B_n))-6-\frac{6}{n}, spr(L(B_n))+2+\frac{2}{n}\}\\
& \leq & ||LI(G_{n+1,\,n})||_{F_3}\\
& = & n+2-\frac{2}{n}
\end{eqnarray*}
with equality if and only if $B_n\cong G_{n+1,\,n}$.
This completes the proof. $\Box$
\end{proof}

Based on the conjecture of Guan et al. \cite{GZW}, we present the following conjecture on the uniqueness of the extremal graph.

\begin{conjecture}
Among all connected graphs with $n$ and $m$ edges $n\leq m \leq 2n-3$,
the $G_{m,\, n}$ is the unique graph with maximal value of $||LI(G)||_{F_2}$ and $||LI(G)||_{F_3}$.
\end{conjecture}

\begin{picture}(300,65)
\put(65,40){\circle*{3}}
    \put(65,40){\line(-5,4){23}}
    \put(65,40){\line(-5,-4){23}}
    \put(65,40){\line(1,0){20}}
    \put(85,40){\circle*{3}}
     \put(85,40){\line(1,0){20}}
     \put(105,40){\circle*{3}}
     \put(113,37){$\cdots$}
    \put(135,40){\circle*{3}}
     \put(135,40){\line(1,0){20}}
     \put(155,40){\circle*{3}}
     \put(155,40){\line(1,0){20}}
    \put(175,40){\circle*{3}}
    \put(43,58){\circle*{3}}
    \put(43,58){\line(0,-1){36}}
    \put(43,22){\circle*{3}}
    \put(108,5){\footnotesize $U_n^2$}

    \put(265,40){\circle*{3}}
    \put(265,40){\line(-5,4){23}}
    \put(265,40){\line(-5,-4){23}}
    \put(265,40){\line(1,0){20}}
    \put(285,40){\circle*{3}}
     \put(285,40){\line(1,0){20}}
     \put(305,40){\circle*{3}}
     \put(313,37){$\cdots$}
    \put(335,40){\circle*{3}}
     \put(335,40){\line(1,0){20}}
     \put(355,40){\circle*{3}}
     \put(355,40){\line(1,0){20}}
    \put(375,40){\circle*{3}}
    \put(375,40){\line(5,4){23}}
    \put(375,40){\line(5,-4){23}}
    \put(397,58){\circle*{3}}
    \put(397,58){\line(0,-1){36}}
    \put(397,22){\circle*{3}}
    \put(243,58){\circle*{3}}
    \put(243,58){\line(0,-1){36}}
    \put(243,22){\circle*{3}}
    \put(318,5){\footnotesize $B_n^1$}

\put(130,-15){Fig. 5.1 \quad Graphs $U_n^2$ and $B_n^1$. }
\end{picture}

\vskip 3mm
{\bf Acknowledgment}

The authors are grateful to the anonymous referees for their valuable comments which result in an improvement of the
earlier version of this paper.

\small {

}

\end{document}